\documentclass[11pt]{amsart}

\author[C.~Sanna]{Carlo Sanna$^\dagger$}
\thanks{$\dagger\,$C.~Sanna is a member of GNSAGA of INdAM and of CrypTO, the group of Cryptography and Number~Theory of Politecnico di Torino}
\address{\parbox{\linewidth}{
Politecnico di Torino, Department of Mathematical Sciences\\
Corso Duca degli Abruzzi 24, 10129 Torino, Italy\\[-8pt]}}
\email{carlo.sanna.dev@gmail.com}



\title{Additive bases and Niven numbers}

\usepackage{amsmath}
\usepackage{amssymb}
\usepackage{amsthm}
\usepackage{geometry}
\usepackage{color}
\usepackage{hyperref}
\hypersetup{colorlinks=true}
\usepackage{enumerate}
\DeclareMathAlphabet{\curly}{U}{rsfs}{m}{n}

\newtheorem{thm}{Theorem}[section]
\newtheorem{cor}[thm]{Corollary}
\newtheorem{lemma}[thm]{Lemma}
\theoremstyle{remark}
\newtheorem{rem}{Remark}[section]

\renewcommand{\Re}{\operatorname{Re}}
\DeclareMathOperator{\HRH}{HRH}

\begin{document}

\maketitle

\begin{abstract}
Let $g \geq 2$ be an integer.
A natural number is said to be a \emph{base-$g$ Niven number} if it is divisible by the sum of its base-$g$ digits.
Assuming Hooley's Riemann Hypothesis, we prove that the set of base-$g$ Niven numbers is an additive basis, that is, there exists a positive integer $C_g$ such that every natural number is the sum of at most $C_g$ base-$g$ Niven numbers.
\end{abstract}

\section{Introduction}

One of the principal problems of additive number theory is to determine, given a set of natural numbers $\mathcal{A}$, if there exists a positive integer $k$ such that every natural number
(resp., every sufficiently large natural number) is the sum of at most $k$ elements of $\mathcal{A}$. 
In such a case, $\mathcal{A}$ is said to be an \emph{additive basis} (resp., an \emph{asymptotic additive basis}) \emph{of order} $k$.

Probably the most famous result of additive number theory is Lagrange's theorem, proved by Lagrange in 1770, which says that the set of perfect squares is an additive basis of order $4$.
More generally, Waring's problem asks whether the set of perfect $h$-th powers is an additive basis, which was answered in the affirmative by Hilbert in 1909.
Furthermore, in 1937, Vinogradov proved that every sufficiently large odd number is the sum of three prime numbers, which implies that the set of prime numbers is an asymptotic basis of order $4$.
For an introduction to these classic results see, e.g., Nathanson's book~\cite{MR1395371}.

Let $g \geq 2$ be an integer. 
Recently, some authors considered additive basis of sets of natural numbers whose base-$g$ representations are restricted in certain ways.
For example, Cilleruelo, Luca, and Baxter~\cite{MR3834696}, improving a result of Banks~\cite{MR3458332}, proved that, for $g \geq 5$, the set of natural numbers whose base-$g$ representations are palindrome is an additive basis of order~$3$.
Rajasekaran, Shallit, and Smith~\cite{MR4081148} showed that the same is true for $g = 3,4$ but not for $g = 2$; and they proved that the binary palindromes are an additive basis of order $4$.
Moreover, Madhusudan, Nowotka, Rajasekaran, and Shallit~\cite{MR3853973} proved that the set of natural numbers whose binary representations consist of two identical repeated blocks is an asymptotic basis of order $4$, while Kane, Sanna, and Shallit~\cite{MR4055207} gave a generalization regarding $k$ repeated blocks. 
For other results of this kind see also \cite{MR3835513, MR3855936}.

A natural number is said to be a \emph{base-$g$ Niven number} if it is divisible by the sum of its \mbox{base-$g$} digits.
De~Koninck, Doyon, and K\'atai~\cite{MR1957109}, and (independently) Mauduit, Pomerance, and S\'ark\"ozy~\cite{MR2166377}, proved that the number of base-$g$ Niven numbers not exceeding $x$ is asymptotic to $c_g x / \!\log x$, as $x \to +\infty$, where $c_g > 0$ is an explicit constant (see~\cite{MR2605530} for a generalization).
Also, De~Koninck, Doyon, and K\'atai~\cite{MR1988644, MR2439527} studied gaps between Niven numbers, and runs of consecutive Niven numbers.

Our result is the following:

\begin{thm}\label{thm:main}
Let $g \geq 2$ be an integer.
Assuming Hooley's Riemann Hypothesis, we have that the set of base-$g$ Niven numbers is an additive basis.
\end{thm}

Hooley's Riemann Hypothesis for an integer $a$ ($\HRH(a)$ for short) states that, for all squarefree positive integers $m$, the Dedekind zeta function $Z_K$ of the number field $K := \mathbb{Q}(\zeta_m,\!\! \sqrt[m]{a})$, where $\zeta_m$ is a primitive $m$-th root of unity, satisfies the Riemann hypothesis, that is, if $Z_K(s) = 0$ for some $s \in \mathbb{C}$, with $0 < \Re(s) < 1$, then $\Re(s) = 1/2$.
We assumed $\HRH(g_0)$, where $g_0$ is an integer depending on $g$, to use some deep results of Frei, Koymans, and Sofos~\cite{FKS2019} (Theorem~\ref{thm:FKS1} and Theorem~\ref{thm:FKS2} below) concerning sums of three prime numbers with prescribed primitive roots.
Except for that, our proof of Theorem~\ref{thm:main} employs only elementary methods.

Finding an unconditional proof of Theorem~\ref{thm:main} and determining the order of the additive basis of the set of base-$g$ Niven numbers are two natural problems.
We checked that every natural number not exceeding $10^9$ is the sum of at most two base-$10$ Niven numbers.

A related problem stems from considering the multiplicative analog of Niven numbers.
A natural number is said to be a \emph{base-$g$ Zuckerman number} if it is divisible by the product of its \mbox{base-$g$} digits.
De~Koninck and Luca~\cite{MR2298113} (see also~\cite{MR3734412} for the correction of a numerical error in \cite{MR2298113}), and Sanna~\cite{MR4181552} gave upper and lower bounds for the number of base-$g$ Zuckerman numbers not exceeding $x$.
In particular, there are at least $x^{0.122}$, and at most $x^{0.717}$, base-$10$ Zuckerman numbers not exceeding $x$, for every sufficiently large $x$.
A question is whether the set of base-$g$ Zuckerman numbers is an additive basis.
We checked that every natural number $n \neq 106$ not exceeding $10^9$ is the sum of at most four base-$10$ Zuckerman numbers.

\section{Preliminaries}

Throughout this section, let $g \geq 2$ be a fixed integer.
For every positive integer $n$, there are uniquely determined $d_1, \dots, d_\ell \in \{0, \dots, g - 1\}$, with $d_\ell \neq 0$, such that $n = \sum_{i=1}^\ell d_i g^{i - 1}$.
We~let $[n]_g := d_1, \dots, d_\ell$ (a string), $s_g(n) := \sum_{i = 1}^\ell d_i$, and $\ell_g(n) := \ell$.
Moreover, for two strings $a$ and $b$, we write $a \leq b$ if $a$ is a substring of $b$, and we let $a \mid b$ be the concatenation of $a$ and $b$.

We begin with two simple lemmas.

\begin{lemma}\label{lem:split}
Let $n$ and $s_1, \dots, s_v$ be positive integers such that $s_g(n) = s_1 + \cdots + s_v$ and $s_v > (g - 2)(v - 1)$.
Then there exist positive integers $n_1, \dots, n_v$ such that $[n]_g = [n_1]_g \mid \cdots \mid [n_v]_g$ and $|s_g(n_i) - s_i| \leq (g - 2)(v - 1)$ for $i=1,\dots,v$.
\end{lemma}
\begin{proof}
If $v = 1$ then the claim follows by picking $n_1 := n$.
Hence, assume that $v \geq 2$.
Let $n = \sum_{j=1}^\ell d_j g^{j - 1}$, where $d_1, \dots, d_\ell \in \{0, \dots, g - 1\}$, with $d_\ell \neq 0$.
We construct $n_1, \dots, n_v$ in the following way: $n_1 := \sum_{j=1}^{\ell_1} d_j g^{j-1}$, where $\ell_1$ is the minimal integer in $[1, \ell]$ such that $\sum_{j=1}^{\ell_1} d_j \geq s_1$; then $n_2 := \sum_{j = \ell_1 + 1}^{\ell_2} d_j g^{j - \ell_1 - 1}$, where $\ell_2$ is the minimal integer in ${(\ell_1, \ell]}$ such that $\sum_{j=\ell_1+1}^{\ell_2} d_j \geq s_2$; and so on, up to $n_{v-1} := \sum_{j = \ell_{v-2} + 1}^{\ell_{v-1}} d_j g^{j-\ell_{v-2}-1}$, where $\ell_{v-1}$ is the minimal integer in ${(\ell_{v-2}, \ell]}$ such that $\sum_{j=\ell_{v-2} + 1}^{\ell_{v-1}} d_j \geq s_{v-1}$;
Finally, $n_v := \sum_{j=\ell_{v-1} + 1}^\ell d_j g^{j - \ell_{v-1} - 1}$.
From this construction, it follows that $s_i \leq s_g(n_i) \leq s_i + g - 2$ and, by induction, that
\begin{equation}\label{equ:sjdisums}
\sum_{j \,=\, i + 1}^v s_j -(g-2)i \leq \sum_{j\,=\,\ell_i + 1}^\ell d_j \leq \sum_{j \,=\, i + 1}^v s_j,
\end{equation}
for $i=1,\dots,v-1$.
In fact, the first inequality in~\eqref{equ:sjdisums} and $s_v \geq (g-2)(v-2)$ ensure that each $\ell_1, \dots, \ell_{v-1}$ is well defined.
Moreover, \eqref{equ:sjdisums} with $i=v-1$ yields $s_v - (g - 2)(v - 1) \leq s_g(n_v) \leq s_v$.
In particular, $n_v > 0$ since $s_v > (g - 2)(v - 1)$.
Lastly, by the minimality of each $\ell_i$, it follows that $d_{\ell_i} \neq 0$, so that $[n_i]_g = d_{\ell_{i - 1}+1}, \dots, d_{\ell_i}$, for $i=1,\dots,v$, where $\ell_0 := 0$ and $\ell_v := \ell$.
Consequently, $[n]_g = [n_1]_g \mid \cdots \mid [n_v]_g$ and the proof is complete.
\end{proof}

\begin{lemma}\label{lem:join}
Let $n$ and $n_1, \dots, n_v$ be positive integers such that $[n]_g = [n_1]_g \mid \cdots \mid [n_v]_g$ and $n_i$ is the sum of $t_i$ base-$g$ Niven numbers for $i=1,\dots,v$.
Then $n$ is the sum of $t_1 + \cdots + t_v$ base-$g$ Niven numbers.
\end{lemma}
\begin{proof}
The claim follows easily after noticing that if $m$ is a base-$g$ Niven number then $g^i m$ is a base-$g$ Niven number for every integer $i \geq 0$.
\end{proof}

The next theorem is a result of additive combinatorics due to Dias~da~Silva and Hamidoune~\cite{MR1272299}.
For every integer $h \geq 1$ and every subset $\mathcal{A}$ of an additive abelian group, let $h^\wedge \mathcal{A}$ denote the set of the sums of $h$ pairwise distinct elements of $\mathcal{A}$, that is, $h^\wedge \mathcal{A} := \big\{\!\sum_{a \in \mathcal{A}^\prime} a : \mathcal{A}^\prime \subseteq \mathcal{A},\, |\mathcal{A}^\prime| = h\big\}$.

\begin{thm}\label{thm:hwedgeA}
Let $h$ be a positive integer, let $p$ be a prime number, and let $\mathcal{A} \subseteq \mathbb{F}_p$.
Then
\begin{equation*}
\left|h^\wedge \mathcal{A}\right| \geq \min\!\big\{p, h|\mathcal{A}| - h^2 + 1 \big\} .
\end{equation*}
In particular, if $|\mathcal{A}| \geq \tfrac1{h}(p - 1) + h$ then $h^\wedge \mathcal{A} = \mathbb{F}_p$.
\end{thm}

The next lemma shows that every positive integer whose sum of digits satisfies certain properties can be written as the sum of a bounded number of Niven numbers.
Hereafter, write $g = g_0^{2^u}$, where $g_0 \geq 2$ is a nonsquare integer and $u \geq 0$ is an integer.

\begin{lemma}\label{lem:root}
If $n$ is a positive integer such that:
\begin{enumerate}[(i)]
\item $s_g(n) = p + h$ for a prime number $p$ and an integer $h \in [4g, 8g]$;
\item $g_0$ is a primitive root modulo $p$; and
\item $s_g(n) > \max\!\big\{\tfrac{g - 1}{3}\, \ell_g(n), 140g^3 \big\}$;
\end{enumerate}
then $n$ is the sum of at most $8g + 1$ base-$g$ Niven numbers.
\end{lemma}
\begin{proof}
Put $\ell := \ell_g(n)$, $s := s_g(n)$, and write $n = \sum_{i = 0}^{\ell - 1} d_i g^i$, where $d_0, \dots, d_{\ell-1} \in \{0, \dots, g-1\}$.
Also, let $\mathcal{I} := \big\{i \in \{0, \dots, \ell - 1\big\} : d_i \neq 0\}$ and $\mathcal{I}^\prime := \{i \bmod t : i \in \mathcal{I}\}$, where $t$ is the multiplicative order of $g$ modulo $p$.
By (ii) and recalling that $g = g_0^{2^u}$, we get that $t \geq (p - 1) / 2^u > (p - 1) / g$.
Hence,
\begin{equation*}
|\mathcal{I}| = \sum_{i^\prime \in\, \mathcal{I}^\prime} |\{i \in \mathcal{I} : {i \equiv i^\prime} \!\!\!\!\mod t\}| < |\mathcal{I}^\prime| \left(\frac{\ell}{t} + 1\right) < |\mathcal{I}^\prime| \left(\frac{g \ell}{p - 1} + 1\right).
\end{equation*}
Since $g$ modulo $p$ has order $t$, letting $\mathcal{A} := \{g^i \!\!\!\mod p : i \in \mathcal{I}\} \subseteq \mathbb{F}_p$ we have
\begin{align*}
|\mathcal{A}| &= |\mathcal{I}^\prime| > \frac{|\mathcal{I}|}{\frac{g \ell}{p - 1} + 1} \geq \frac{\frac{s}{g - 1}}{\frac{g \ell}{p - 1} + 1} > \frac{\frac{s}{g - 1}}{\frac{3gs}{(g - 1)(p - 1)} + 1} \\
 &= \frac{(p - 1)s}{3gs + (g - 1)(p - 1)} > \frac{p - 1}{4g - 1} > \frac{p - 1}{4g} + 8g \geq \frac{p - 1}{h} + h ,
\end{align*}
where we used the inequalities $|\mathcal{I}| \geq \tfrac{s}{g - 1}$, $\ell < \tfrac{3}{g - 1} s$, $s > p - 1 > 128g^3$; of which the last three are consequences of (i) and (iii).
Hence, Theorem~\ref{thm:hwedgeA} yields that $h^\wedge \mathcal{A} = \mathbb{F}_p$.
In~particular, there exists $\mathcal{J} \subseteq \mathcal{I}$ such that $|\mathcal{J}| = h$ and $\sum_{i \in \mathcal{J}} g^i \equiv n \pmod p$.
As a consequence, letting $m := n - \sum_{i \in \mathcal{J}} g^i$, it follows easily that $s_g(m) = s - h = p$ and $m \equiv 0 \pmod p$, so that $m$ is a base-$g$ Niven number.
Thus $n = m + \sum_{i \in \mathcal{J}} g^i$ is the sum of $h + 1$ base-$g$ Niven numbers and, recalling that $h \leq 8g$, the proof is complete.
\end{proof}

For all integers $q > 0$ and $r$, let $\mathcal{S}_{q,r}$ be the set of of positive integers $n$ such that:
\begin{enumerate}[(S1)]
\item\label{ite:S1} $s_g(n) \equiv r \pmod q$;
\item\label{ite:S2} for all positive integers $m$ such that $[m]_g \leq [n]_g$ and $\ell_g(m) \geq 36 \log \ell_g(n)$, we have $s_g(m) > \tfrac{g - 1}{3} \ell_g(m)$.
\end{enumerate}
Recall that the \emph{lower asymptotic density} of a set of positive integers $\mathcal{A}$ is defined as the limit infimum of $|\mathcal{A} \cap [1, x]| / x$, as $x \to +\infty$.

\begin{lemma}\label{lem:Aqr}
Let $q > 0$ and $r$ be integers.
Then $\mathcal{S}_{q,r}$ has positive lower asymptotic density.
\end{lemma}
\begin{proof}
Let $\ell > q$ be an integer and let $n$ be a uniformly distributed random integer in $\{0, \dots, g^\ell - 1\}$.
Then $n = \sum_{i = 1}^\ell d_i g^{i-1}$, where $d_1, \dots, d_\ell$ are independent uniformly distributed random variables in $\{0,\dots,g - 1\}$.
On the one hand, we have
\begin{align*}
P_1 &:= \Pr\big[\ell_g(n) = \ell \text{ and } s_g(n) \equiv r \!\!\!\!\pmod q\big] = \Pr\left[d_\ell \neq 0 \text{ and } \sum_{i \,=\, 1}^\ell d_i \equiv r \!\!\!\pmod q \right] \\
 &= \sum_{s \,=\, 0}^{q - 1} \Pr\left[d_\ell \neq 0 \,\text{ and }\!\! \sum_{i \,=\, \ell - q + 1}^\ell d_i \equiv r - s \!\!\!\pmod q \right] \cdot \Pr\left[\,\sum_{i \,=\, 1}^{\ell - q} d_i \equiv s \!\!\!\pmod q \right] \\
 &\geq \frac1{g^q} \sum_{s \,=\, 0}^{q - 1} \Pr\left[\,\sum_{i \,=\, 1}^{\ell - q} d_i \equiv s \!\!\!\pmod q \right] \geq \frac1{g^q} .
\end{align*}
On the other hand, by Hoeffding's inequality~\cite[Theorem~2]{MR144363}, we have
\begin{equation*}
\Pr\left[\sum_{j \,=\, 1}^k d_{i + j} \leq \frac{g - 1}{3} k\right] \leq e^{-k / 18} ,
\end{equation*}
for all $k \in \{1, \dots, \ell\}$ and $i \in \{0, \dots, \ell - k\}$.
Hence, letting $y := 36\log \ell$, we get
\begin{align*}
P_2 &:= \Pr\left[\exists m \in \mathbb{N} \text{ s.t. } [m]_g \leq [n]_g, \; \ell_g(m) \geq y,\, s_g(m) \leq \tfrac{g-1}{3}\ell_g(m)\right] \\
&\leq \sum_{\substack{y \,\leq\, k \,\leq\, \ell \\[1pt] i \,\in\, \{0,\dots,\ell-k\}}} \Pr\left[\sum_{j \,=\, 1}^k d_{i + j} \leq \frac{g - 1}{3} k\right] \leq \ell \sum_{k \,\geq\, y} e^{-k/18} \leq \frac{\ell \, e^{-y / 18}}{1 - e^{-1/18}} < \frac{20}{\ell} \to 0 ,
\end{align*}
as $\ell \to +\infty$.
Therefore, for every sufficiently large $x > 0$, letting $\ell$ be the greatest integer such that $g^\ell \leq x$, we obtain
\begin{equation*}
\frac{\big|\mathcal{S}_{q,r} \cap [1, x]\big|}{x} > \frac{\big|\mathcal{S}_{q,r} \cap {[1, g^\ell)}\big|}{g^{\ell + 1}} \geq \frac{P_1 - P_2}{g} > \frac1{2g^{q + 1}} .
\end{equation*}
Hence, $\mathcal{S}_{q,r}$ has positive lower asymptotic density.
\end{proof}

The next result is an easy consequence of an important theorem of Schnirelmann.

\begin{thm}\label{thm:schnirelmann}
Let $\mathcal{A}$ be a set of positive integers such that $1 \in \mathcal{A}$.
If $\mathcal{A}$ has positive lower asymptotic density, then $\mathcal{A}$ is an additive basis.
\end{thm}
\begin{proof}
Since $1 \in \mathcal{A}$ and ${\displaystyle\liminf_{n \,\to\, +\infty}}\frac{|\mathcal{A} \,\cap\, [1, n]|}{n} > 0$, it follows that ${\displaystyle\inf_{n \,\geq\, 1}}\frac{|\mathcal{A} \,\cap\, [1, n]|}{n} > 0$, that is, $\mathcal{A}$ has positive \emph{Schnirelmann density}.
Consequently, by Schnirelmann's Theorem~\cite[Theorem~7.7]{MR1395371}, it follows that $\mathcal{A}$ is an additive basis.
(Note that in~\cite{MR1395371} they say that $\mathcal{A}$ is a \emph{basis of finite order} if there exists a positive integer $k$ such that every natural number is the sum of \emph{exactly} $k$ elements of $\mathcal{A}$, 
and that~\cite[Theorem~7.7]{MR1395371} has to be applied to $\mathcal{A} \cup \{0\}$.)
\end{proof}

The following deep results of Frei, Koymans, and Sofos~\cite[Theorem~1.1 and Theorem~1.7]{FKS2019} are crucial to the proof of Theorem~\ref{thm:main}.

\begin{thm}\label{thm:FKS1}
Let $\mathbf{a} = (a_1, a_2, a_3) \in \mathbb{Z}^3$ such that no $a_i$ is $-1$ or a square. 
Assuming $\HRH(a_i)$ for $i = 1, 2, 3$, we have
\begin{equation*}
\sum_{\substack{n \,=\, p_1 + p_2 + p_3 \\[2pt] a_i \text{ primitive~root~mod } p_i, \\[2pt]\text{for }i\,=\,1,2,3.}} \, \prod_{i\,=\,1}^3 \log p_i \sim \curly{A}_{\mathbf{a}} (n)\, n^2 \quad \text{ as } n \to +\infty ,
\end{equation*}
with an explicit factor $\curly{A}_{\mathbf{a}} (n) \geq 0$ that satisfies $\curly{A}_{\mathbf{a}}(n) \geq A_\mathbf{a}$, whenever $\curly{A}_{\mathbf{a}} (n) > 0$, where $A_\mathbf{a} > 0$ is a constant depending only on $\mathbf{a}$.
\end{thm}

\begin{thm}\label{thm:FKS2}
Let $a \neq -1$ be a nonsquare integer.
Then $\curly{A}_{(a,a,a)}(n) > 0$ if and only if $(n \bmod  420) \in \mathcal{R}_a$, where $\mathcal{R}_a$ is a nonempty set of residues modulo $420$ depending only on $a$.
\end{thm}

As a consequence, we obtain the following:

\begin{cor}\label{cor:FKS}
Let $a \neq -1$ be a nonsquare integer and let $\delta \in (0,1)$.
Assuming $\HRH(a)$, for every sufficiently large natural number $n$ such that $(n \bmod 420) \in \mathcal{R}_a$ there exist prime numbers $p_1, p_2, p_3 > n^\delta$ such that $n = p_1 + p_2 + p_3$ and $a$ is a primitive root modulo each of $p_1, p_2, p_3$.
\end{cor}
\begin{proof}
The claim follows easily from Theorem~\ref{thm:FKS1} and Theorem~\ref{thm:FKS2} by noticing that
\begin{align*}
\sum_{\substack{n \,=\, p_1 + p_2 + p_3 \\[1pt] p_1 \,\leq\, n^{\delta}}} \prod_{i \,=\, 1}^3 \log p_i \leq \sum_{\substack{p_1 \,\leq\, n^\delta \\[1pt] \!\!\!p_2 \,\leq\, n}} (\log n)^3 \leq n^{\delta + 1} (\log n)^3 = o(n^2) ,
\end{align*}
as $n \to +\infty$.
\end{proof}

\section{Proof of Theorem~\ref{thm:main}}

Let $g \geq 2$ be an integer, write $g = g_0^{2^u}$, where $g_0 \geq 2$ is a nonsquare integer and $u \geq 0$ is an integer, and assume that $\HRH(g_0)$ holds.
Put $q := 420$ and $r := r^\prime + 18g$, where $r^\prime$ is any fixed element of $\mathcal{R}_{g_0}$, and let $\mathcal{A} := \mathcal{S}_{q,r}$ so that, thanks to Lemma~\ref{lem:Aqr}, $\mathcal{A}$ has positive lower asymptotic density.
By Theorem~\ref{thm:schnirelmann}, we have that $\mathcal{A} \cup \{1\}$ is an additive basis.

Now, in order to prove Theorem~\ref{thm:main}, it suffices to show that every sufficiently large (depending only on $g$) element of $\mathcal{A}$ is the sum of a bounded number (depending only on $g$) of base-$g$ Niven numbers.
Let $n \in \mathcal{A}$ be sufficiently large, and let $\ell := \ell_g(n)$, $s := s_g(n)$, and $s^\prime := s - 18g$.

Clearly, $\ell \to +\infty$ as $n \to +\infty$.
In particular, $\ell \geq 36 \log \ell$ for every sufficiently large $n$.
Hence, from~(\hyperref[ite:S2]{S2}) with $m = n$, we get that $s > \tfrac{g - 1}{3} \ell$ and $s^\prime > \tfrac{g - 1}{3}\ell - 18g$.
Consequently, in what follows, we can assume that $\ell, s, s^\prime$ are sufficiently large.

By~(\hyperref[ite:S1]{S1}), we have $s^\prime \equiv s - 18g \equiv r - 18g \equiv r^\prime \pmod q$, so that $(s^\prime \bmod 420) \in \mathcal{R}_{g_0}$ and $s^\prime$ is sufficiently large.
Hence, by Corollary~\ref{cor:FKS}, there exist prime numbers $p_1, p_2, p_3 > \sqrt{s^\prime}$ such that $s^\prime = p_1 + p_2 + p_3$ and $g_0$ is a primitive root modulo each of $p_1, p_2, p_3$.

As a consequence, $s = s_1 + s_2 + s_3$ where $s_i := p_i + 6g$ for $i=1,2,3$.
Hence, by Lemma~\ref{lem:split}, there exist positive integers $n_1, n_2, n_3$ such that $[n]_g = [n_1]_g \mid [n_2]_g \mid [n_3]_g$ and $|s_g(n_i) - s_i| \leq 2(g - 2)$ for $i=1,2,3$.
In particular, $s_g(n_i) = p_i + h_i$ for some integer $h_i \in [4g, 8g]$.
Note that $[n_i]_g \leq [n]_g$ and
\begin{equation*}
\ell_g(n_i) \geq \frac{s_g(n_i)}{g - 1} > \frac{p_i}{g - 1} > \frac{\sqrt{s^\prime}}{g - 1} \geq \frac{\sqrt{\tfrac{g-1}{3}\ell - 9g}}{g - 1} > 36 \log \ell .
\end{equation*}
Therefore, from~(\hyperref[ite:S2]{S2}) it follows that $s_g(n_i) > \tfrac{g - 1}{3}\ell_g(n_i)$.

Thus we have proved that each $n_i$ satisfies the hypotheses of Lemma~\ref{lem:root}, and consequently each $n_i$ is the sum of at most $8g + 1$ base-$g$ Niven numbers.
Then, from Lemma~\ref{lem:join}, it follows that $n$ is the sum of at most $24g + 3$ base-$g$ Niven numbers.

The proof is complete.

\begin{rem}\label{rmk:strong}
An inspection of the proof of Theorem~\ref{thm:main}, in particular Lemma~\ref{lem:root}, shows that, actually, we proved a stronger result:
Assuming $\HRH(g_0)$, the union of $\{g^i : i=0,1,\dots\}$ and the set of base-$g$ Niven numbers $m$ such that $p = s_g(m)$ is a prime number and $g_0$ is a primitive root modulo $p$ is an additive basis.
\end{rem}

\section{Acknowledgements}

The computational resources were provided by \texttt{hpc@polito} (\texttt{http://www.hpc.polito.it}).
The author thanks Daniele Mastrostefano (University of Warwick) for suggestions that improved the paper.

\end{document}